\documentclass{amsart}
\usepackage{amssymb,amsthm}
\usepackage{amsfonts}
\usepackage{enumerate}
\usepackage{amsmath}
\usepackage[all,cmtip]{xy}
\usepackage{graphicx}
\usepackage[utf8]{inputenc}
\usepackage[english]{babel}
\usepackage{color}
\usepackage{hyperref}
\usepackage{pgf}
\usepackage{colortbl,xcolor}
\usepackage{multicol}
\usepackage{geometry}
\usepackage{bigints}

\theoremstyle{plain}

\newtheorem{thm}{Theorem}[section]
\newtheorem{defi}[thm]{Definition}
\newtheorem{cor}[thm]{Corollary}
\newtheorem{lemma}[thm]{Lemma}
\newtheorem{prop}[thm]{Proposition}
\newtheorem{example}[thm]{Example}
\newtheorem{rmk}[thm]{Remark}

\newcommand{\C}{\mathbb C}

\newcommand{\Id}{\textrm{Id}}
\newcommand{\Max}{\mathfrak M}

\newcommand{\R}{\mathbb R}

\newcommand{\Z}{\mathbb Z}

\definecolor{lightgrey}{cmyk}{0,0,0,0.30}
\definecolor{darkgrey}{cmyk}{0,0,0,0.70}
\definecolor{purple}{cmyk}{0.45,0.86,0,0}
\definecolor{darkblue}{cmyk}{1.7,0.7,0.3,0}
\definecolor{lightblue}{cmyk}{0.6,0.2,0.1,0}
\definecolor{midblue}{cmyk}{1.1,0.3,0.2,0}

\newcommand{\Ver}{Vert\,}
\newcommand{\Cone}{Cone\,}

\newcommand{\trans}{^{\mathbf{t}}}

\DeclareMathOperator{\M}{\mathrm{M}}
\DeclareMathOperator{\N}{\mathrm N}

\DeclareMathOperator{\Rnnoneg}{(\R_{\geq 0})^n}
\DeclareMathOperator{\Znnoneg}{(\Z_{\geq 0})^n}
\DeclareMathOperator{\vertice}{\mathbf{\nu}}

\begin{document}	
\title{On Archimedean Zeta Functions and Newton Polyhedra}

\author{Fuensanta Aroca}
	\address{Universidad Nacional Aut\'{o}noma de M\'{e}xico, 
		Unidad Cuernavaca. 
		A.P.273-3 C.P. 62251.
		Cuernavaca, MOR.
		M\'{e}xico.}	
	\email{fuen@matcuer.unam.mx}

\author{Mirna G\'omez-Morales}
	\address{Universidad Nacional Aut\'{o}noma de M\'{e}xico, 
		Unidad Cuernavaca. 
		A.P.273-3 C.P. 62251.
		Cuernavaca, MOR.
		M\'{e}xico.}	
	\email{m.gomez@im.unam.mx}

\author{Edwin Le\'{o}n-Cardenal}
	\address{CONACYT -- Centro de Investigaci\'{o}n en Matem\'{a}ticas, 
		Unidad Zacatecas.
		Quantum Ciudad del Conocimiento.
		Avenida Lasec, Andador Galileo Galilei,
		Manzana 3 Lote 7. C.P. 98160.
		Zacatecas, ZAC.
		M\'{e}xico.}
	\email{edwin.leon@cimat.mx}

\date{}


\begin{abstract}
	Let $f$ be a polynomial function over the complex numbers and let $\phi$ be a
	smooth function over $\C$ with compact support. When $f$ is non-degenerate with respect to its Newton polyhedron, we give an explicit list of candidate poles for the complex local zeta function attached to $f$ and $\phi$. The provided list is given just in terms of the normal vectors to the supporting hyperplanes of the Newton polyhedron attached to $f$. More precisely, our list does not contain the candidate poles coming from the additional vectors required in the regular conical subdivision of the first orthant, and necessary in the study of local zeta functions through resolution of singularities.
	
	Our results refine the corresponding results of Varchenko and generalize the results of Denef and Sargos in the real case, to the complex setting.
	
\end{abstract}
\maketitle

	\section{Introduction}\label{Sec1}
	
	Archimedean local zeta functions were introduced by Gel'fand and Shilov in the 50's in \cite{Gel-Shi}. 
	Take $K=\mathbb{R}$ or $\mathbb{C}$ and take $f(x)=f(x_1,\ldots,x_n)\in K[x_1,\ldots,x_n]$. Let $\phi$ be a
	smooth function with compact support in $K^n$. The local zeta function attached to $(f,\phi)$ is the parametric integral
	\[
	Z_{\phi}(s,f)=\int\limits_{K^{n}\smallsetminus f^{-1}(0)}\phi(x)\ |f(x)|_K^s\ |dx|,
	\]
	for $s\in\mathbb{C}$ with $\operatorname{Re}(s)>0$, where $|dx|$ is the Haar
	measure on $K^{n}$. For uniformity reasons we will use for $a\in\mathbb{C}$, the convention: $|a|_K=|a|_\mathbb{C}^2$, where $|a|_\mathbb{C}$ is the standard complex norm. 
	
	It is not difficult to show that $Z_{\phi}(s,f)$ converges on the half plane $\{s\in\mathbb{C}\ ;\  \operatorname{Re}(s)>0 \}$ and defines a holomorphic function there. Furthermore, $Z_{\phi}(s,f)$ admits a meromorphic continuation to the whole complex plane. This was proved by Bernstein and Gel'fand in \cite{Ber-Gel}, then independently by Atiyah in \cite{Ati}; both proofs make use of Hironaka's theorem on resolution of singularities. Later, Bernstein \cite{Ber} gave a proof by using what is called nowadays $D-$module theory. The main motivation of Gel'fand behind this problem was that the meromorphic continuation of $Z_{\phi}(s,f)$ implies the existence of fundamental solutions for differential operators with constant coefficients, see e.g. \cite[Section 5.5]{Igusa}.
	
	Since those days the theory of local zeta functions has evolved considerably due to multiple connections with many fields of mathematics, such as number theory, representation theory, and singularity theory among others. For instance, it is known that the poles of $Z_{\phi}(s,f)$ are integer shifts of the roots of the Bernstein--Sato polynomial of $f$, and therefore, by the classical results of Malgrange, the poles induce eigenvalues of the complex monodromy of $f$. These relations give some light about the difficult task that represents in general to compute the poles of $Z_{\phi}(s,f)$. However, there is a special case in which some of these invariants can be computed effectively: the non degenerated case. 
	
	Varchenko shows in \cite{Var} that when $f$ is assumed to be non degenerate with respect to its Newton polyhedron $NP(f)$, the set of candidate poles of $Z_{\phi}(s,f)$ can be described explicitly in terms of the normal vectors to the supporting hyperplanes of $NP(f)$. This description is very useful to study, for instance, real oscillatory integrals. Indeed Varchenko proves that the poles of the local zeta function control the asymptotic behavior of the oscillatory integrals associated to $f$ and $\phi$, see \cite[\textsection 1]{Var}.
	
	Roughly speaking Varchenko's idea is to attach a Newton polyhedron $NP(f)$ to the function $f$ and then define a non degeneracy condition with respect to $NP(f)$. Then one may construct a toric variety associated to the Newton polyhedron, and use the well known toric resolution of singularities to give a list of candidate poles for $Z_{\phi}(s,f)$. Toric resolution of singularities requires a  regular fan subordinated to $NP(f)$ (see Section \ref{Sec2} for the corresponding definitions) and it turns out that the extra rays required to obtain such a regular fan give rise to fake candidate poles, see Remark \ref{Refining fans}. In the case $K=\mathbb{R}$, Denef and Sargos proved in \cite{Den-Sar} that the poles coming from those extra rays can be discarded, thus reducing the  list of candidate poles. In \cite{Le-Ve-Zu} this proof is extended for real zeta functions of analytic mappings. In this work we prove the analogue of the results of Denef and Sargos for the case $K=\mathbb{C}$, thus providing a much shorter list of candidate poles that can be read off directly from the geometry of the Newton polyhedron of $f$. 
	
	The method employed by Denef and Sargos relies heavily on the particular structure of the real field and at this point we do not know if their methods can be extended to the complex case. Here we propose a different approach that is based on toric embeddings as in \cite{Ar-GM-Sha}, rather than toric resolution of singularities. The toric embeddings that we obtain are in general multi--valued maps, i.e. coverings (see Section \ref{Momomial transf}), so we use them as change of variables for coverings (see Lemma \ref{Cambio de variables para recubrimientos}) in order to analyze the integral that appears as the pull-back of $Z_{\phi}(s,f)$ to the toric variety. Then we are left with some monomial integrals for which the poles are easily described, see Lemma \ref{Monomial Integrals}. The main result of this article is Theorem \ref{Main Thm}, which says that when $f$ is non-degenerate with respect to $NP(f)$, and $\phi$ has a small enough support, then the poles of the complex zeta function $Z_{\phi}(s,f)$ are contained in certain arithmetic progressions that are given just in terms of the faces of $NP(f)$. Moreover, Theorem \ref{Main Thm} can be adapted in a straightforward manner to reprove  Th\'eor\`eme 1.1, Th\'eor\`eme 1.2 and Th\'eor\`eme 6.1 in \cite{Den-Sar}.
	
	A natural task that remains to be undertaken is the generalization of our results to the case when $f$ is a holomorphic function. The main difficulty that we have encountered in that case is the generalization of the `transversality' condition of Section \ref{Section transversality} for holomorphic $f$, see Remark \ref{Generalization holomorphic case}.
	
	We would like to emphasize that our approach to the study of complex zeta functions avoids the use of a toric resolution of singularities, thus the relevant information about the poles of $Z_\phi(s,f)$ relies just in the geometry of the Newton polyhedron of $f$, more precisely in the geometry of its normal fan (see Remark \ref{Refining fans}). This point of view is on the same line that the work of Gilula in \cite{Gil} for real oscillatory integrals. Perhaps our methods, combined with those of \cite{Gil} may give a better estimation and/or asymptotic expansions of complex oscillatory integrals like the ones studied in \cite{Vasi}. The subject of real  oscillatory integrals constitutes a very active area of research, see (among many others) \cite{Ber-Gel,Ch-Ka-No,Co-Gr-Pr,Den-Sar,Gel-Shi,Gil,Gi-Gr-Xi,Gree,Gre,Igusa,Ik-Mu,Ka-No2,Ka-No1,Le-Ve-Zu,Ph-Ste-Stu,Var,Vasi,Ve-Zu}. We believe that our results may be of some interest in this community. Note also that the kind of problems that we address here is also an object of study for topological zeta functions and its relatives, see for example \cite{Es-Le-Ta,Le-vPr}.
	
	The authors would like to thank Hussein Mortada for providing us with the reference of Theorem 3.4.11 in \cite{Cox-Lit-Sch}. 
	We are also grateful to Wilson Z\'u\~niga-Galindo for many fruitful conversations during the preparation of this work. Some of the ideas behind this work were developed during a research visit of the second and third named authors at Universidad Aut\'onoma de Aguascalientes, we want to express our gratitude for their hospitality. Lastly, we would like to thank the anonymous referees for useful suggestions and comments, specially for the observation that appears in Remark \ref{Weakening the non degeneracy condition}.
	
\section{Fans, Monomial transformations and Toric Varieties}\label{Sec2}
In this section, we review some basic results about toric geometry, such as cones, fans and toric varieties. The material presented in this section can be found for instance in \cite{Ewald,Fulton,Oka}. 

\subsection{Cones and polyhedral fans.}
The rational polyhedral cone generated by $u^{(1)},\ldots , u^{(k)}\in \Z^n$ is the set $\langle u^{(1)},\dots,u^{(k)}\rangle =\{t_1 u^{(1)}+\cdots+ t_k u^{(k)}; \; t_i\in \R_{\geq 0},\; i=1, \ldots, k\}\subset \R^n$. By ${\mathcal L} (\sigma)$ we will denote the  minimal linear subspace  containing the cone $\sigma$. The \textit{dimension} of $\sigma$, denoted by $\dim(\sigma)$, is the dimension of ${\mathcal L} (\sigma)$ and  the \textit{relative interior} of $\sigma$, denoted by $Int_{rel} \sigma$, is the interior of $\sigma$ as a subset of ${\mathcal L} (\sigma )$. 

A rational polyhedral cone is said to be \textit{strongly convex} if it does not contain any non-trivial linear subspace.  Note that a cone contained in the first orthant is strongly convex. 

A set of generators $\{u^{(1)},\dots,u^{(n)}\}$ of a rational cone $\langle u^{(1)},\ldots, u^{(k)}\rangle$ can be chosen to be primitive, i.e. such that for any $i$, $\gcd_j(u_j^{(i)})=1$. If, furthermore, the set $\{u^{(1)}, \ldots, u^{(k)}\}$ is minimal it would be called the set of \textit{vertices} of the cone $\langle u^{(1)},\ldots, u^{(k)}\rangle$. For a strongly convex cone, its set of vertices is uniquely determined. We will denote by $\sigma=\Cone(u^{(1)}, \ldots, u^{(k)})\subset \R^n$ the rational convex cone with vertices $u^{(1)}, \ldots, u^{(k)}\in\Z^n$; or, for simplicity $\sigma=\Cone(\N)$, where $\N$ is the $k\times n$ matrix having the vertices $u^{(1)}, \ldots, u^{(k)}$ of $\sigma$ as columns. A strongly convex rational polyhedral cone $\sigma=\Cone(u^{(1)}, \ldots, u^{(k)})$ is said to be \textit{simplicial} if the group ${\mathcal L} (\sigma )\cap \Z^n$ has rank $k$ and, furthermore, \textit{regular} if ${\mathcal L} (\sigma )\cap \Z^n$ is generated by the vertices of $\sigma$.

The \textit{dual} $\sigma^{\vee}$ of a cone $\sigma$ is the polyhedral cone given by $\sigma^{\vee}:= \{ v \in \R^n ; v \cdot u \geq 0,\, \text{for all}\, u\in \sigma \}$, where $u\cdot v$ stands for the inner product of the vectors $u=(u_1,\ldots ,u_n)$ and $v=(v_1,\ldots ,v_n)$. Note that given an $n$-dimensional rational simplicial cone $\sigma=\Cone(\N)$ there is a matrix $\M\in \mathcal{M}(n,\Z)$ with columns $v^{(1)}, \ldots, v^{(n)}\in\Z^n$, such that 
\begin{equation} \label{gensdeldualdeunracional}
{\sigma}^{\vee}= \Cone( v^{(1)}, \ldots, v^{(n)}) \;\; \mbox{and} \; \,\M\N\trans=\begin{pmatrix}
\lambda_1 & 0 & \ldots &0\\
0 & \lambda_2&\ddots  &0\\
\vdots & \ddots&\ddots  &\ddots\\
0  & 0&\dots & \lambda_n 
\end{pmatrix}= \N \trans \M, 
\end{equation} 
with positive integers $\lambda_1, \dots, \lambda_n$. Here $\N\trans$ stands for the transpose of the matrix $\N$. 
Set $L_{\M}$ to be the linear map 
\begin{align*}
L_{\M} : \R^n &\to \R^n\\
x &\mapsto \M x.
\end{align*}
The following properties of $L_{\M}$ are straightforward.
\begin{lemma} \label{L sub M del dual es el primer ortante}
	\begin{enumerate}[(i)]
		\item $L_{\M}\left( {(\R_{\geq 0})}^n\right)=\Cone (\M )$.\\
		\item $L_{\N\trans}(\sigma^{\vee}\cap\Z^n)=\{\N\trans \cdot \beta\ ; \ \beta\in \sigma^{\vee}\cap\Z^n\}\subset (\Z_{\geq 0})^n.$\\
		\item  If $\M, \N\in  \mathcal{M}(n,\Z)$ satisfy $\N\trans \M=\Lambda$, for some diagonal matrix $\Lambda$ with values $\lambda_1, \dots, \lambda_n\in\R_{>0}$ along its diagonal, then $L_{\N\trans}\left(\Cone \left( M\right)\right)=\Rnnoneg$.
	\end{enumerate}
\end{lemma}
There is a simple test to known when a given $n-$dimensional cone $\sigma=\Cone(N)$ is regular or not: $\sigma$ is regular if and only if $\N\in GL (n,\Z)$ is an unimodular matrix. In this case \eqref{gensdeldualdeunracional} holds for $\M=(\N^{-1})\trans$ and $\lambda_1=\cdots=\lambda_n=1$. 

A finite collection $\Sigma=\{\sigma\}_{\sigma\subset \R^n}$ of cones is called a \textit{polyhedral fan} if every face of a cone in $\Sigma$ is a cone in $\Sigma$, and the intersection of any two cones $\sigma, \tau \in \Sigma$ is a face of both $\sigma$ and $\tau$. The \textit{support} $|\Sigma|$ of a polyhedral fan $\Sigma$ is the union of its cones. A polyhedral fan is said to be \textit{simplicial} if all of its cones are simplicial and \textit{regular} if all of its cones are regular. In a simplicial polyhedral fan $\Sigma$, a cone $\tau$ is a face of another cone $\sigma$ if and only if the set of vertices of $\tau$ is a subset of the set of vertices of $\sigma$. For a polyhedral fan $\Sigma$, its \textit{set of vertices}, denoted $\Ver(\Sigma)$, is the union of vertices of cones in $\Sigma$. 

\begin{rmk}\label{Refining fans}
	Given a polyhedral fan $\Sigma$, it is always possible to find a simplicial (respectively, regular) fan $\Sigma^\prime$ refining $\Sigma$, that is, such that every cone in $\Sigma$ is union of cones in $\Sigma^\prime$, and with the property that every cone in $\Sigma^\prime$ is simplicial (respectively, regular) (see for example \cite{Fulton}). The procedure of finding a simplicial fan refining $\Sigma$ can be done without changing the set of vertices in $\Sigma$, but obtaining a regular fan  $\Sigma^\prime$ refining $\Sigma$ would require, in general, the introduction of new vertices. 
\end{rmk}

\subsection{Monomial transformations defined for integer matrices.}\label{Momomial transf}
Given a non-singular matrix $\M\in \mathcal{M}(n,\Z)$ with columns $u^{(1)},u^{(2)},\ldots,u^{(n)}\in \Z^n$, denote by $\psi_{\M}$ the 
morphism given by 
\begin{align*}
\psi_{\M} : (\C^{*})^n & \longrightarrow  (\C^{*})^n\\
z& \longmapsto \zeta,
\end{align*}
where $\C^*:=\C\setminus\{ 0\}$ and $\zeta_k=\prod_{j=1}^{n}z_j^{u_j^{(k)}}.$ This morphism is rational on $\C^n$ and it is bi-rational if $\det M=\pm1$. When $\det M=\pm d$, the morphism $\psi_{\M}$ is a $d$-fold covering, see e.g. \cite{Oka}. Furthermore, for any matrices $\M, \N\in \mathcal{M}(n,\Z)$ we have
\[\psi_{\M} \circ \psi_{N} = \psi_{{N}M }; \;  \text{ while for }  
\M\in GL(n,\Z)\;  \text{ we have } \;(\psi_{\M} )^{-1} = \psi_{\M^{-1}}.
\]

\subsection{The toric variety associated to a simplicial polyhedral fan.}\label{T Sigma}

Let us recall the construction of the toric variety associated to a simplicial polyhedral fan $\Sigma$ of $\mathbb{R}^n$
with support in the first orthant. This process can be found, for example in \cite{Ewald}. Denote by $\Max (\Sigma)$ the set of maximal dimension cones in $\Sigma$, i.e. $n-$dimensional cones. Take a simplicial cone $\sigma=\Cone(N)\in\Max (\Sigma)$, with $\sigma^{\vee}=\Cone(v^{(1)}, \ldots, v^{(n)})$ satisfying \eqref{gensdeldualdeunracional}. Since the first orthant is contained in $\,\sigma^{\vee}$, we have that the canonical basis of $\mathbb{R}^n:\ \{e^{(1)}, \ldots, e^{(n)}\},$ is contained in $\sigma^{\vee}\cap \Z^n$. Choose $w^{(1)}, \ldots, w^{(\ell)}\in\Z^n$ so that the set $\{v^{(1)}, \ldots, v^{(n)}$, $e^{(1)}, \ldots, e^{(n)}$, $w^{(1)}, \ldots, w^{(\ell)}\}\subset\Z^n$ generates the semi-ring $\,\sigma^{\vee}\cap \Z^n$. Consider the morphism 
\begin{align*}
\varphi_{\sigma}: (\C^{*})^n &\to (\C^{*})^n\times(\C^{*})^n\times (\C^{*})^\ell\subset \C^{2n+\ell}\\\label{phi-sigma}
y & \mapsto (y^{v^{(1)}},\ldots,y^{v^{(n)}}, y^{e^{(1)}},\dots,y^{e^{(n)}}, y^{w^{(1)}},\dots,y^{w^{(\ell)}}).
\end{align*}

Set $U_{\sigma}:=\textrm{Image} (\varphi_{\sigma})$. The closure $T_{\sigma}:=\overline{U_{\sigma}}\subset \C^{2n+\ell}$ is the affine variety associated to the cone $\sigma$. Using local coordinates $(x,y,z)$ in $\C^n\times\C^n\times \C^\ell$, we have:
\begin{equation*}
T_{\sigma}=\left\{(x,y,z)\in \C^n\times\C^n\times \C^\ell;(x,y,z)^{(\alpha,\beta,\gamma)}=(x,y,z)^{(\alpha^{\prime},\beta^{\prime},\gamma^{\prime})}\right\},
\end{equation*}
where $\alpha, \alpha^{\prime}, \beta, \beta^{\prime}\in (\Z_{\geq 0})^n; \gamma,\gamma^{\prime}\in (\Z_{\geq 0})^\ell$ and they satisfy  \[\sum^n_{i=1} (\alpha_i -\alpha^{\prime}_i)v^{(i)}+\sum^n_{j=1} (\beta_j-\beta^{\prime}_j)e^{(j)} +\sum^\ell_{k=1} (\gamma_k-\gamma^{\prime}_k) w^{(k)}=0.\]

The toric variety $T_{\Sigma}$ associated to $\Sigma$ is constructed by glueing the affine toric varieties associated to each of its cones in such a way that $\textrm{Image} (\varphi_{\sigma})= \textrm{Image}(\varphi_{\sigma^{'}})$ for $\sigma,\sigma^{'}\in \Sigma$. This construction of affine toric varieties is well known, furthermore one may show that $T_{\Sigma}$ is a Hausdorff variety. 

The following lemma will be used in the proof of our main result. Its proof is a straightforward calculation.

\begin{lemma}\label{Prop:JacdepsiA}
	Let $A\trans\in \mathcal{M}(n,\Z)$ be the matrix with columns $u^{(1)},\ldots,u^{(n)}$, where $u^{(i)} =(u^{(i)}_1, \dots, u^{(i)}_n)\trans\in (\Z_{\geq 0})^n$.
	Then the Jacobian of the morphism $\psi_{A}$ at $x$ is given by
	\begin{equation*}
	\det \,D (\psi_{A})_x=\det A \cdot x_1^{||u^{(1)}||-1}\cdots x_n^{||u^{(n)}||-1},
	\end{equation*}
	where $||u^{(j)}||:=u^{(j)}_1+\cdots +u^{(j)}_n$. 
\end{lemma}

\subsection{Toric modification}
In this section we will introduce some morphisms defined over $T_\sigma$ (or more generally over $T_\Sigma$) that will be used in the computation of $Z_\phi(s,f)$. 
For each $\sigma\in\Max (\Sigma)$, set $\pi_{\sigma}$ to be the morphism 
\begin{align*}
\pi_{\sigma} : T_{\sigma} \subset \C^n\times\C^n\times \C^\ell &\longrightarrow \C^n\\
(x,y,z)\quad & \longmapsto y.
\end{align*}

\noindent
It can be verified that the morphisms $\pi_{\sigma}$ are compatible with the glueing, defining a morphism $\displaystyle \pi:T_{\Sigma}\to \C^n$ such that for each $\sigma\in\Max (\Sigma)$, $\displaystyle \pi|_{T_{\sigma}}= \pi_{\sigma}$. Since the support of $\Sigma$ is the first orthant, the morphism 
$\displaystyle \pi:T_{\Sigma}\to \C^n$ is a proper and bi-rational morphism (see for example \cite[Thm. 3.4.11]{Cox-Lit-Sch}) called the \textit{toric modification} associated to $\Sigma$. This morphism $\pi$ is bi-regular in the complement of the coordinate hyperplanes.

Take $\sigma = Cone (\N)\in\Max (\Sigma)$. Let  $\M$ be the matrix such that $\sigma^\vee =\Cone (\M)$ and let $\Lambda$ be the diagonal matrix $\N\trans \M$. Let $v^{(1)}, \ldots, v^{(n)}$ be the columns of $\M$, let $w^{(1)}, \ldots, w^{(\ell)}\in\Z^n$ be such that the set 
\[\{v^{(1)}, \ldots, v^{(n)}, e^{(1)}, \ldots, e^{(n)}, w^{(1)}, \ldots, w^{(\ell)}\}\subset\Z^n\]
generates the semi-ring $\,\sigma^{\vee}\cap \Z^n$, and let $W$ be the $(\ell\times n)-$matrix that has $w^{(1)}, \ldots, w^{(\ell)}$ as columns.  

Now, consider the morphism
\begin{align*}
\varphi_{\sigma}: (\C^{*})^n &\longrightarrow \C^{2n+\ell}\\
y & \longmapsto (y^{v^{(1)}},\ldots,y^{v^{(n)}}, y^{e^{(1)}},\dots,y^{e^{(n)}}, y^{w^{(1)}},\dots,y^{w^{(\ell)}}),
\end{align*}
and the projections $\rho_x:\C^{2n+\ell}\to \C^n$ and $\rho_y:\C^{2n+\ell}\to \C^n$ defined in local coordinates $(x,y,z)\in\C^n\times\C^n\times \C^\ell$ by $\rho_x(x,y,z)=x\in \C^n$ and $\rho_y(x,y,z)=y\in \C^n$. 

The diagram 

\begin{large}
	\begin{equation}\label{Fig:DiagramCompleto}
	\xymatrix{
		&& &\; \;\C^n \ar[rr]^f  
		&&\C\\
		&& & & &\\
		(\C^{*})^n \ar@{^{(}->}[rr]^{\varphi_{\sigma}} \ar@/^1pc/@{^{(}->}[rrruu]^{\textrm{Id}\; \quad} 
		\ar@/_1pc/@{^{(}->}[rrrdd]^{\psi_{\M}\; \quad} &
		&T_{\sigma}\subset \C^{2n+\ell} \;\ar@/_/[rdd]_{\rho_x} \ar@/_0.6pc/[ruu]^{\rho_y}&\quad \C^n \ar[dd]^{\psi_{\Lambda}} \ar[uu]_{\psi_{\N\trans}}& (\C^{*})^n \ar@/^/[ldd]^{\psi_{\M}} \ar@/_0.5pc/@{_{(}->}[luu]_{\;\textrm{Id}}\\ 
		& & & &\\
		& & &\C^n  & &
	}
	\end{equation}
\end{large}

\noindent
commutes. In particular, the restriction $\pi_{\sigma}:=\pi|_{T_{\sigma}}={\rho_y}|_{T_{\sigma}}$ of the toric modification associated to $\Sigma$ satisfies the equalities $\rho_x\circ 
\varphi_{\sigma}=\psi_{\M}$ and $\pi_{\sigma}\circ 
\varphi_{\sigma}=\Id|_{(\C^\ast)^n}$. 

\begin{defi}\label{La parametrizacion de la carta asociada a sigma}
	Take $\sigma = Cone (\N)\in\Max (\Sigma)$. Let  $\M$ be the matrix such that $\sigma^\vee =\Cone (\M)$ and let $\Lambda$ be the diagonal matrix $\N\trans \M$. 
	Let $\Phi_\sigma$ be given by
	\begin{align*}
	\C^n &\longrightarrow T_\sigma\\
	x & \longmapsto (\psi_\Lambda (x), \psi_{\N\trans}(x), \psi_{{\N\trans} W} (x)).
	\end{align*}
\end{defi}
The following diagram illustrates the role of $\Phi_\sigma$ in the precedent construction of morphisms,
\begin{large}
	\begin{equation}\label{Diagrama de covering}
	\xymatrix{
		&& &\; \;\C^n \ar[rr]^f  
		&&\C\\
		&& & & &\\
		& 
		&T_{\sigma} \;\ar@/_/[rdd]_{\rho_x}  \ar@/_0.6pc/[ruu]^{\pi_\sigma}&\ar@/_/[l]_{\Phi_\sigma}\C^n  \ar[dd]^{\psi_{\Lambda}} \ar[uu]_{\psi_{\N\trans}}& \\ 
		& & & &\\
		& & &\C^n  & &
	}
	\end{equation}
\end{large}
Note that $\Phi_\sigma|_{(\C^\ast)^n}=\varphi_{\sigma}\circ\psi_{\N\trans}$ and all the exponents appearing here are positive so we can extend this map to the whole $\C^n$. This shows that $\Phi_\sigma$ is well defined. Also we have that $\Phi_\sigma : {(\C^\ast)}^n\longrightarrow T_\sigma^\ast$ is a finite covering, moreover, it is a local diffeomorphism. 
\subsection{Partitions of Unity}\label{Particiones de unidad}
We will use a partial resolution of singularities to decompose the integral $Z_\phi(s,f)$ into monomial integrals. A key ingredient in our proof is the Theorem of existence of partitions of the unity.

Recall that, for any topological space $X$, given a function $\chi: X\longrightarrow \C$, the \emph{support of} $\chi$ is the closed set
\[
Supp (\chi):= \overline{\{x\in X; \chi(x)\neq 0\}}.
\]
\begin{defi}\label{definicion de C infinito}
	Let $T_\Sigma$ be a toric variety, and consider $U$ a connected open subset of $T_\Sigma$. Let $\chi :U\longrightarrow \C$ be a function and take a point $x$ in $U$. Take $\sigma\in \Max (\Sigma)$ such that $x\in T_\sigma\subset \C^{2n+\ell}$. We will say that {\bf $\chi$ is $\mathcal{C}^\infty$ at $x$} if there exists $V\subset \C^{2n+\ell}$ a neighbourhood of $x$ in $\C^{2n+\ell}$, and a $\mathcal{C}^\infty-$function $\Psi: V\longrightarrow \C$ such that $\Psi|_{V\cap U}=\chi|_{V\cap U}$.
\end{defi}
The fact that the definition does not depend on the affine chart $T_\sigma$ follows from the fact that the glueings defining $T_\Sigma$ are regular morphisms.

\begin{defi}
	Let $X$ be a topological space and let $\{ U_\alpha\}_{\alpha\in\Lambda}$ be a collection of open sets with $X=\bigcup_{\alpha\in\Lambda}U_\alpha$. {\bf A partition of unity subordinated to the cover  $\{ U_\alpha\}_{\alpha\in\Lambda}$ } is a collection of continuous functions $\xi_\alpha : X\longrightarrow \C$, such that $\sum_{\alpha\in\Lambda} \xi_\alpha (x) =1$ and the support of each $\xi_\alpha$ is contained in $U_\alpha$.
\end{defi}
It is well known that partitions of unity subject to a covering exist for the complex $n$-dimensional space. Moreover, the partitions can be chosen such that the $\xi_\alpha$'s are $\mathcal{C}^\infty$. Then, there exist $\mathcal{C}^\infty-$partitions of unity in the sense of Definition \ref{definicion de C infinito} for toric varieties.

	\section{Newton Polyhedra}\label{Seccion NP and non deg}
\subsection{The Newton polyhedron of a function}\label{Subsec:DualfanMirna}
In this section we recall the construction and some properties of the Newton polyhedron of a polynomial and the corresponding toric variety associated to its normal fan. Most of the statements in this section may also be found in \cite{Ar-GM-Sha}.

In this section we will take $K=\mathbb{R}$ or $\mathbb{C}$. Let $f$ be a $K$-polynomial function, such that $f(\underline{0})=0$. We write 
\begin{equation}\label{escritura del polinomio f}
f(x)=\sum_{\mu\in (\Z_{\geq 0})^n} 	a_{\mu}x^{\mu}\quad\text{for}\quad x=(x_1,\ldots,x_n)\in K^n.
\end{equation} 
The \textit{support} or \textit{set of exponents} of $f$ is the set 
\[\varepsilon (f):= \{\mu\in {(\Z_{\geq 0})^n};  a_{\mu}\neq 0\}.\]
The \textit{restriction} of $f$ to a subset  $F\subset (\R_{\geq 0})^n$ is defined as
\[f|_F=\sum_{\mu\in \varepsilon (f)\cap F} a_{\mu}x^{\mu}.\]
\begin{example}\label{Restringir a esta cara es evaluar en 0}
	Take $\{e^{(1)},\ldots,e^{(k)}\}$ from the canonical basis of $\R^n$ for $1\leq k \leq n$. If $\langle e^{(1)},\ldots,e^{(k)} \rangle$  denotes the cone generated by these vectors, then 
	\[f|_{\left\langle e^{(1)},\dots ,e^{(k)}\right\rangle} = f(x_1,\ldots ,x_k,0,\ldots ,0).\] Therefore $f(x)=\sum_{\mu\in (\Z_{\geq 0})^n} 	a_{\mu}x^{\mu}$	can be written as
	\[	 f= f|_{\left\langle e^{(1)},\dots ,e^{(k)}\right\rangle} + \sum_{i=k+1}^n x_i \widetilde{f_i},\] 
	for some $K$-polynomial functions $\widetilde{f}_{k+1},\dots, \widetilde{f}_n$.
\end{example}
\begin{defi}
	The \textit{Newton polyhedron} of $f$ is  the following convex hull 
	\begin{equation*}
	NP(f):=Conv(\{\mu + (\R_{\geq 0})^n ; \mu \in \varepsilon (f)\} ) \subset (\R_{\geq 0})^n.
	\end{equation*}
\end{defi}
\begin{rmk}\label{Si solo tiene un vertice es un monomio por una unidad}
	The polyhedron $NP(f)$ has only one vertex $\vertice\in (\Z_{\geq 0})^n$ if and only if $f(x)=x^{\vertice}h_{\vertice}(x)$ where $h_{\vertice}(x)\in K [[x_1,\ldots ,x_n]]$ satisfies $h_{\vertice}(\underline{0})\neq 0$. 
\end{rmk}

Let $\mathcal{H}$ be the hyperplane given by $\left\{  x\in\mathbb{R}^{n}\  ;\  a\cdot x =b\right\}  $. The hyperplane $\mathcal{H}$ determines two closed half-spaces:
\[
\mathcal{H}^{+}:=\left\{  x\in\mathbb{R}^{n}\mathbf{;}\, a\cdot x\geq b\right\}
\quad\text{and}\quad
\mathcal{H}^{-}:=\left\{  x\in\mathbb{R}^{n}\mathbf{;} \,a\cdot x\leq b\right\}  .
\]
We say that $\mathcal{H}$ is \textit{a supporting hyperplane} of $NP(f)$, if $NP(f)\cap \mathcal{H}\neq\emptyset$ and $NP(f)\subset \mathcal{H}^+$ or $NP(f)\subset \mathcal{H}^-$.

A \textit{proper face} of $NP(f)$ is the intersection of the polyhedron with a supporting hyperplane, and the \textit{non-proper} face is the whole $NP(f)$. Faces of dimension $0, 1$, and $n-1$ are called vertices, edges and facets, respectively. 

When $F$ is a face of $NP(f)$, the restriction $ f|_F$ is often denoted by $f_F$ and called the \textit{face function}.

Given a supporting hyperplane $\mathcal{H}$ of $NP(f)$ containing a facet, there exists a vector $u\in \mathbb{Z}^n\setminus\{0\}$ which is orthogonal to $\mathcal{H}$ and is directed into the polyhedron. Such a vector is called an \textit{inward} normal to $\mathcal{H}$. When the vector $u$ is chosen to be primitive, it turns out that every facet of $NP(f)$ has a unique primitive inward vector; the set of such vectors is denoted by  $\mathcal{I}(NP(f))$.

Now, given $\omega\in\Rnnoneg$, the \textit{$\omega$-order} of  $f$ is defined as
\[ \nu_{\omega}(f):=min\{\omega \cdot \mu ; \mu\in \varepsilon(f)\}.\]
Note that $\mathcal{H}_\omega:=\{ x\in\R^n\ ;\ \omega\cdot x =\nu_\omega(f)\}$ is a \textit{supporting hyperplane} for $NP(f)$, and the intersection 
\[F_{\omega} := NP(f)\cap \mathcal{H}_\omega
\] is a face of $NP(f)$ called \textit{the first meet locus} of $\omega$ or the \textit{$\omega$-face}.

\subsection{Dual fans and fans subordinated to $f$.}
We define an equivalence relation on $\Rnnoneg$ by taking
\[\omega\sim \omega^{\prime}\Longleftrightarrow F_\omega=F_{\omega^{\prime}}.\]
In order to describe the equivalence classes of $\sim$ we define \textit{the cone
	associated to} a given face $F$ of $NP(f)$, as \[\sigma_F := \{\omega \in  (\R_{\geq 0})^n\ ;\ F = F_\omega\}.\]
Note that $\sigma_{NP(f)}=\{\underline{0}\}$. The other equivalence classes are described in the following Lemma. 

\begin{lemma}\label{DualFan}\begin{enumerate}[(i)]
		\item Let $F$ be a proper face of $NP(f)$, then
		the topological closure $\overline{\sigma}_{F}$ of
		$\sigma_{F}$ is a rational polyhedral cone and
		\[\overline{\sigma}_{F}=\left\{ \omega \in  (\R_{\geq 0})^n\ ;\ F_\omega \supset F\right\}.\]
		\item Let $F_{1},\ldots,F_{k}$ be the facets of $NP(f)$ containing
		$F$ and let $u^{(1)},\ldots,u^{(k)}\in\mathbb{Z}^{n}\setminus\left\{
		\underline{0}\right\}$ be the inward normal vectors to $F_{1},\ldots,F_{k}$ respectively. Then
		$\sigma_{F}=\left\{\lambda_{1}u^{(1)}+\cdots+\lambda_ku^{(k)}\ ;\ \lambda_{i}\in\mathbb{R}_{>0}\right\}$,
		and  $\overline{\sigma}_{F}=\langle u^{(1)},\dots,u^{(k)}\rangle.$
		\item $\dim\sigma_{F}=\dim\overline{\sigma}_{F}=n-\dim F$.
	\end{enumerate}
\end{lemma}
\begin{defi}
	The collection of cones $\Sigma(f):=\{ \overline{\sigma}_F\ ;\ F \text{ is a face of } NP (f)\}$  is called the \textit{dual fan} of $f$. 
\end{defi}

There is a natural duality between cones in $\Sigma(f)$ and faces of $NP(f)$ given by
\[
F\mapsto \overline{\sigma}_F
\]
and
\[\sigma \mapsto F_\sigma, \]
where $F_\sigma$ is the intersection of the facets having as inward vectors the generators of $\sigma$. 

From the previous construction one has that $\Sigma(f)$ is a polyhedral fan with support $\Rnnoneg$. Moreover, \[\Ver(\Sigma(f))=\mathcal{I}(NP(f)).\]  
We will say that a fan $\Sigma$ \textit{is subordinated to} $f$ if it defines a refinement of $\Sigma(f)$. From now on, we will work with simplicial fans subordinated to $f$.

\begin{prop}\label{El conjunto de exponentes esta contenido en el dual de un buen cono}
	\begin{enumerate}\item $NP(f)=\bigcap_{\vertice\,\mbox{\tiny{vertex of}}\; NP(f)} \left( \vertice+ \overline{\sigma}^{\,\vee}_{\{\vertice\}}\right).$
		\item If $\Sigma$ is a fan subordinated to $f$ and $\sigma\in\Max (\Sigma)$, the face $\vertice := F_\sigma$ of $NP(f)$ is a vertex of $NP (f)$ with $\sigma\subset \overline{\sigma}_{\vertice}$. We have $\varepsilon (f)\subset \vertice+ {\sigma^{\,\vee}}$, thus \[f(x)=x^{\vertice} h_\sigma(x),\]
		where $h_\sigma(x)\in K [[x_1^{\pm 1},\ldots ,x_n^{\pm 1}]]$ satisfies $\underline{0}\in\varepsilon (h_\sigma)\subset \sigma^{\vee}\cap\Z^n$.
		\item Let $\Sigma$ be a simplicial fan subordinated to $f$ and let $\sigma =Cone (\N)\in \Max (\Sigma)$. If $\vertice := F_\sigma$, then
		\[
		L_{\N\trans} (\vertice )=(\nu_{u^{(1)}} (f),\ldots ,\nu_{u^{(n)}} (f))
		\]
		where $u^{(1)} ,\ldots ,u^{(n)}$ are the columns of the matrix $\N$.
	\end{enumerate}
\end{prop}
\begin{proof}
	All the statements follow easily from the given definitions and Lemma \ref{DualFan}.
\end{proof}

\subsection{Newton Polyhedron under monomial modifications}
Let $f$ be a polynomial as in \eqref{escritura del polinomio f}. Given $\mu=(\mu_1,\ldots,\mu_n)\in\Z^n$ and $\M=(u_k^{(j)})\in\mathcal{M}(n,\Z)$, we have ${\psi_{\M} (x)}^\mu =x^{\M\cdot\mu}$. Hence, 
\begin{equation} \label{supportfcomppi}
f\circ \psi_{\M} (x)= \sum_{\mu\in \varepsilon(f)} a_{\mu}x^{L_{\M}(\mu)}\quad \text{and} \quad \varepsilon(f\circ\psi_{\M})= L_{\M} (\varepsilon(f)). 
\end{equation}
Then, for any subset $F\subset {(\R_{\geq 0})}^n$,
\begin{equation} \label{restricciones y modificaciones monomiales}
\left. f \right|_F\circ \psi_{\M} (x) = \left.\left( f\circ \psi_{\M} (x)\right) \right|_{L_M (F)}.
\end{equation}

\begin{prop} \label{El poliedro se convierte en el primer cuadrante trasladado}
	Let $\Sigma$ be a simplicial fan subordinated to the  polynomial $f$, and let $\sigma =Cone (\N)\in \Max (\Sigma)$, then the following assertions hold.
	\begin{enumerate}
		\item The polyhedron $NP (f\circ \psi_{\N\trans})$  has only one vertex. Furthermore, we have  \[f\circ \psi_{\N\trans}(x)=x^{L_{\N\trans}(\vertice)}h(x),\]		
		where $h(x)\in K [x_1,\ldots ,x_n]$ satisfies $h(\underline{0})\neq 0$ and  $L_{\N\trans}(\vertice)\in(\Z_{\geq 0})^n$ with $\vertice$ vertex of $NP (f)$.  
		\item If $\tau$ is the face of $\sigma$ generated by the $i^{th}$ columns of $\N$, with $i\in J\subset \{1,\ldots ,n\}$. Then, we have 
		\[ \left( f\circ\psi_{\N^t}\right) |_{L_{\N\trans}(\vertice)+ \langle e^{(i)}\ ;\ i\notin J\rangle}=f|_{F_{\tau}}\circ \psi_{\N^t}, \]
		where $\vertice$ is the vertex of $NP (f)$ such that $\varepsilon (f)\subset \vertice+ \Cone \left( \N\right)^\vee$. 
	\end{enumerate}
\end{prop}

\begin{proof}
	\begin{enumerate}
		\item It follows from Remark \ref{Si solo tiene un vertice es un monomio por una unidad} and Proposition \ref{El conjunto de exponentes esta contenido en el dual de un buen cono}.
		\item  Let $\M\in \mathcal{M}(n,\Z)$ be the matrix with columns $v^{(1)},\ldots , v^{(n)}$, such that $\sigma^\vee=\Cone \left( \M\right)$ and  $\N\trans  \M\in \mathcal{M}(n,\N)$ is the diagonal matrix with positive values $\lambda_1,\dots,\lambda_n$ along its diagonal. Then,
		\begin{align*}
		F_{\tau} & = NP (f) \cap \left( \vertice + \left( \Cone \left( \N\right)^\vee\cap \tau^\perp\right)\right) \\
		& = NP(f) \cap \left( \vertice + \langle \{ v^{(i)}\}_{i\notin J}\rangle\right)
		\end{align*}
		and hence,
		\begin{equation*}
		f |_{F_{\tau}}=f|_{\vertice + \langle \{ v^{(i)}\}_{i\notin J}\rangle}.
		\end{equation*}
		By \eqref{restricciones y modificaciones monomiales},
		\[
		\left. f \right|_{\vertice + \langle \{ v^{(i)}\}_{i\notin J}\rangle}\circ \psi_{\N\trans} (x) = \left.\left( f\circ \psi_{\N\trans} (x)\right) \right|_{L_{\N\trans} \left(\vertice + \langle \{ v^{(i)}\}_{i\notin J}\rangle\right)}.
		\]
		The result follows from the fact that 
		\begin{multline*}
		L_{\N\trans} \left(\vertice + \langle \{ v^{(i)}\}_{i\notin J}\rangle\right) = L_{\N\trans}\! \left(\vertice\right) + \left\langle  L_{\N\trans} \left( v^{(i)}\right); {i\notin J}\right\rangle\\
		= L_{\N\trans}(\vertice)+ \left\langle \lambda_i e^{(i)};i\notin J\right\rangle
		= L_{\N\trans}(\vertice)+ \left\langle e^{(i)};i\notin J\right\rangle.\qedhere
		\end{multline*}
	\end{enumerate}
	
\end{proof}

Under the hypothesis of the second part of Proposition \ref{El poliedro se convierte en el primer cuadrante trasladado}, the following result holds.
\begin{cor}\label{Como cambian las caras a las que hay que restringir}
	There exists a polynomial $h_{\tau}\in K[x_1,\ldots ,x_n]$ depending only on the variables $x_i$ for $i\notin J$ with $h_{\tau}(\underline{0})\neq 0$ such that 
	\[f|_{F_{\tau}}\circ \psi_{\N\trans}(x)=x^{L_{\N\trans}(\vertice)}h_{\tau}(x).\]
\end{cor}
\begin{rmk}\label{Generalization holomorphic case}
	It is very likely that the results in this section can be generalized to the case when $f$ is a holomorphic function. This generalization would provide a way of generalizing Theorem \ref{Main Thm} to the holomorphic setting. So far, such a generalization has escaped from the attempts of the authors, but we believe that further work in this direction could be of some interest.
\end{rmk}

	\section{Non-degeneracy Condition}

\begin{defi}\label{nondegerate}
	Let $f(x)$ be a polynomial, such that $f(\underline{0})=0$. We say that $f$ is non--degenerate with respect to a face $F\subseteq NP(f)$ if the system of equations
	$$\left\{  f_{F}(x)=0,\nabla f_{F}(x)=0\right\}$$ has no
	solutions in $(K^\ast)^n$. 
	
	We say that $f$ is non--degenerate with respect to its Newton polyhedron if it is non--degenerate with respect to any face of $NP(f)$. 
\end{defi}
It seems that the non-degeneracy condition was first proposed by Arnol'd in \cite{Arn}, where he uses it to classify critical points of functions. See also \cite{Be-Ku-Ho,Ho77,Ho78,Ho83} and \cite{Kou}.

\begin{prop}\label{NoDeg se preserva bajo transf matricial}
	Let $\Sigma$ be a simplicial fan subordinated to the  polynomial $f$ and let $\sigma=\Cone \left( \N\right)\in\Max (\Sigma)$.
	If the function $f$ is non-degenerate with respect to its Newton polyhedron then $f\circ \psi_{\N\trans}$ is also non-degenerate with respect to $NP (f\circ \psi_{\N\trans})$. 
\end{prop}

\begin{proof}
	By the first part of Proposition \ref{El poliedro se convierte en el primer cuadrante trasladado}, the polyhedron $NP (f\circ \psi_{\N\trans})$ has only one vertex $L_{\N\trans}(\vertice)$ and its faces are of the form $E_J:={L_{\N\trans}(\vertice)+ \left\langle e^{(i)}\ ;\ i\notin J\right\rangle}$, for some $J\subset \{1,\ldots ,n\}$, where $\vertice$ is the vertex of $NP (f)$ such that $\varepsilon (f)\subset \vertice+ \Cone \left( \N\right)^\vee$ (that is $\vertice = F_\sigma$). Let $E_J$ be a face of $NP (f\circ \psi_{\N\trans})$ and assume, without loss of generality, that $J=\{1,2,\dots, r\}$. 
	By the second part of Proposition \ref{El poliedro se convierte en el primer cuadrante trasladado},
	\begin{equation*}
	\left( f\circ\psi_{\N^t}\right) |_{E_J} =
	f|_{F_{\tau}}\circ \psi_{\N^t} (x_1,\dots,x_n)
	\end{equation*}
	where $\tau$ is the compact face of $\sigma$ generated by the last $n-r$ columns of $\N$. Since $f$ is non-degenerate, $V(f|_{F_{\tau}})$, the set of $K-$zeroes of $f|_{F_\tau}$, does not have singularities in the coordinate hyperplanes. The result now follows from the fact that $\psi_{\N^t}$ restricted to ${(\C^*)}^n$ is a local diffeomorphism from ${(\C^*)}^n$ to ${(\C^*)}^n$.
\end{proof}

\begin{lemma}\label{monomioporhNoDegimplicahNoDeg}
	Let $g(x), h(x)\in K [x_1,\ldots ,x_n]$ be such that $h(\underline{0})\neq 0$ and
	\begin{equation*}
	g(x)=x^{\alpha}h(x)
	\end{equation*}
	for some $\alpha\in(\Z_{\geq 0})^n$. If $g$ is non-degenerate with respect to $NP(g)$, then $h$ is non-degenerate with respect to $NP(h)$.
\end{lemma}

\begin{proof}
	Note that $F\subseteq NP(g)$ is a face of $NP(g)$ if and only if  $F+(-\alpha)$ is a face of $NP(h)$. Suppose that ${h}|_{F+(-\alpha)}(x)=0$ for some $x\in V(x_1\cdots x_n)$. Since ${g}|_{F}(x)=x^{\alpha}{h}|_{F+(-\alpha)}(x)$, it follows that $g|_{F}(x)=0$ and, furthermore, we have
	\begin{equation*}
	\frac{\partial\left({g}|_{F}\right)}{\partial x_i}(x)=x^{\alpha}\frac{\partial\left({h}|_{F+(-\alpha)}\right)}{\partial x_i}(x)+ x^{\alpha-e_i} {h}|_{F+(-\alpha)}(x).
	\end{equation*}
	Since $g$ is non-degenerate with respect to $NP(g)$, there exists $i\in\{1,\ldots,n\}$ such that
	\begin{equation*}
	0\neq \frac{\partial\left({g}|_{F}\right)}{\partial x_i}(x)=x^{\alpha}\frac{\partial\left({h}|_{F+(-\alpha)}\right)}{\partial x_i}(x),
	\end{equation*}
	which finishes the proof.
\end{proof}

\subsection{Transversality}\label{Section transversality}

Given a point $P\in K^n$ and $r$ polynomials $h_1, \ldots ,h_r$, with $r\leq n$ and $h_i(P)=0$ we say that $V(h_1), \ldots ,V(h_r)$ intersect \textit{transversally} at $P$, when they are all smooth at $P$ and the dimension of the linear subspace generated by $\nabla h_1 (P), \ldots ,\nabla h_r(P)$ is $r$. We say that $V(h_1), \ldots ,V(h_r)$ intersect \textit{transversally}, when they intersect transversally at every common zero $P$ of $V(h_1), \ldots ,V(h_r)$.  
\begin{prop}\label{Si hay transversalidad hay un buen cambio de coordenadas}
	If $V(h_1), \ldots ,V(h_r)$ intersect transversally at $P$, then there exists a neighbourhood $U$ of $P$, a neighbourhood $U'$ of $\underline{0}$  and a diffeomorphism $\eta: (U',\underline{0})\longrightarrow (U,P)$ such that $h_i\circ\eta=x_i$ for $i=1,\ldots ,r$.
	
	And the determinant of the Jacobian of $\eta$ is a $\mathcal{C}^\infty(K)-$function that is different from zero on $U'$.
\end{prop}

\begin{proof}
	Let $i_{r+1},\ldots , i_n\in \{1,\ldots ,n\}$ be such that  \[\{ \nabla h_1 (P), \ldots ,\nabla h_r(P), e^{(i_{r+1})},\ldots,e^{(i_n)}\},\] is a base of $K^n$.
	
	Consider the morphism 
	\[
	\begin{array}{cccc}
	H:	& (K^n, P)
	& \longrightarrow
	& (K^n ,\underline{0})\\
	& x	& \mapsto
	& \left( h_1(x),\ldots ,h_r(x), \left(x_{i_{r+1}}-P_{i_{r+1}}\right),\ldots ,\left( x_{i_n}-P_{i_n}\right)\right).
	\end{array}
	\]
	The determinant of the Jacobian of $H$ is different that zero at $P$ then, by the Inverse Function Theorem, there exists a smooth function $H^{-1}: (K^n ,\underline{0})\longrightarrow (K^n, P)$ such that $H\circ H^{-1}$ is the identity. Taking $\eta := H^{-1}$ we have the result.
\end{proof}

\begin{prop}\label{No se anula en cero mas no degenerada implica transversal}
	Let $h\in K[x_1,\ldots, x_n]$ be a non-degenerate polynomial with respect to its Newton polyhedron, with $h(\underline{0})\neq 0$. The variety $V (h)$ intersects transversally with the coordinate hyperplanes.
\end{prop}

\begin{proof}
	If $n=1$, the coordinate hyperplane corresponds to the origin and $h$ does not vanish at $0$. For $n\geq 2$, let $z=(z_1,\ldots, z_n)\in h^{-1}(0)\cap V(x_1\cdots x_n)$. Since $h(\underline{0})\neq 0$, there exists $i\in \{1,\dots,n\}$ such that  $z_i\neq 0$ and, since $z\in V(x_1\cdots x_n)$, there exists $j\in \{1,\dots,n\}$ such that $z_j=0$. Assume, without loss of generality, that $z_1\cdots z_r\neq 0$ and $z_{k}=0$ for $k\in \{r+1,\dots,n\}$.
	
	In what follows we will show that $V(x_{r+1}),\ldots ,V(x_n)$ and $V(h)$ intersect transversally at $z$. Since $h(\underline{0})\neq 0$, the origin is the only vertex of $NP (h)$ and  
	$\langle e^{(1)},\dots, e^{(r)}\rangle$ is a face of $NP (h)$. Now, $h|_{\langle e^{(1)},\ldots, e^{(r)}\rangle}$ is a polynomial in the variables $x_1,\ldots,x_r$, thus
	\[
	\frac{\partial \left(h|_{\langle e^{(1)},\ldots, e^{(r)}\rangle}\right)}{\partial x_i}= 0\quad\text{for}\quad i\in \{r+1,\ldots ,n\}.
	\] 
	Recall from Example \ref{Restringir a esta cara es evaluar en 0} that $h=h|_{\langle e^{(1)},\dots, e^{(r)}\rangle}+\sum_{j=r+1}^n x_j\widetilde{h}_j$, for some polynomials 
	$\widetilde{h}_{r+1},\dots, \widetilde{h}_n\in K[x_1,\ldots,x_n]$. Therefore
	\[\frac{\partial h}{\partial x_i}=\frac{\partial \left(h|_{\langle e^{(1)},\dots, e^{(r)}\rangle}\right)}{\partial x_i}+\sum_{j=r+1}^n x_j\frac{\partial \widetilde{h}_j}{\partial x_i},\] when $i\in \{1,\dots,r\}$. 
	Evaluating at $z=(z_1,\ldots,z_r,0,\ldots,0)$  we get
	\[\frac{\partial h}{\partial x_i}(z)=\frac{\partial \left(h|_{\langle e^{(1)},\dots, e^{(r)}\rangle}\right)}{\partial x_i}(z).\]
	On the other hand, the non--degeneracy condition over $h$ implies that for any $\varepsilon\in(K^\ast)^n$ there exists some index $i\in \{1,\ldots ,r\}$ such that
	\[\frac{\partial\left( h|_{\langle e^{(1)},\dots, e^{(r)}\rangle}\right)}{\partial x_i}(\varepsilon ) \neq 0.\] If $(z_{r+1}^\ast,\ldots,z_n^\ast)\in (K^\ast)^{n-r}$, then $(z_1,\ldots,z_r,z_{r+1}^\ast,\ldots,z_n^\ast)\in (K^\ast)^{n}$ and for $i\in \{1,\ldots ,r\}$ we have
	\[
	\frac{\partial h}{\partial x_i}(z)
	=\frac{\partial h|_{\langle e^{(1)},\dots, e^{(r)}\rangle}}{\partial x_i}(z)
	=\frac{\partial h|_{\langle e^{(1)},\dots, e^{(r)}\rangle}}{\partial x_i}(z_1,\ldots,z_r,z_{r+1}^\ast,\ldots,z_n^\ast),
	\]
	so there should be some index $i\in \{1,\ldots ,r\}$ such that $\frac{\partial h}{\partial x_i}(z)\neq 0.$
	
	Finally, since 
	\[\{ (\nabla x_{r+1}) (z),\ldots ,(\nabla x_n) (z), \nabla h (z) \}=\{ e^{(r+1)}, \ldots ,e^{(n)}, (\nabla h) (z)\}\] and the set on the right hand side generates an $(n-r+1)-$dimensional space, we obtain the result.
\end{proof}

\begin{cor}\label{h intersecta transv}
	Take $f\in K[x_1,\ldots,x_n]$ and assume that it is non-degenerate with respect to its Newton polyhedron. Let $\Sigma$ be a simplicial fan subordinated to $f$ and take $\sigma=\Cone (\N)\in \Max (\Sigma )$. Then $f\circ \psi_{\N\trans}= x^\alpha h$, where  $h(\underline{0})\neq 0$ and $\alpha\in\Znnoneg$. In addition, the variety defined by $h$ intersects transversally the hyperplane coordinates. 
\end{cor}

\subsection{Neighbourhoods and coordinates}

\begin{prop} \label{hay buenas coordenadas para puntos del x a la alpha por h si h nos se anula y es no degenerada}
	Take $h(x)\in K[x_1,\ldots,x_n]$ with $h(\underline{0})\neq 0$ and assume that $h$ is non-degenerate with respect to its Newton polyhedron. Let $g:= x^\alpha h$, for some $\alpha= (\alpha_1,\ldots ,\alpha_n)\in\Znnoneg$ and let $z$ be a point in the coordinate hyperplanes. Set $J_z:= \{ i\ ;\ z_i=0\}$  and choose $k\notin J_z$. Then there exists a neighbourhood $U_z$ of $z$, a neighbourhood $\widetilde{U}_z$ of the origin and a diffeomorphism
	\[
	\eta_z: (\widetilde{U}_z,\underline{0})\longrightarrow  (U_z,z) 
	\]
	with ${(\eta_z)}_j(x)=x_j$ for all $j\in J_z$ and ${(\eta_z)}_j(x)\neq 0$ for all $j\notin J_z$ for all $x\in \widetilde{U}_z$; and such that one of the following holds
	\begin{enumerate}[(i)]
		\item $g\circ \eta_z =\prod_{j\in J_z} {x_j}^{\alpha_j}\tilde{h}$ where $\tilde{h}(x)\neq 0$ for all $x\in \widetilde{U}_z$,
		\item $g\circ \eta_z =\prod_{j\in J_z} {x_j}^{\alpha_j}x_k\tilde{h}$ where $\tilde{h}(x)\neq 0$ for all $x\in \widetilde{U}_z$.
	\end{enumerate}
	In addition the determinant of the Jacobian of $\eta_z$ is a $\mathcal{C}^\infty(K)-$function not vanishing on $\widetilde{U}_z$.
\end{prop}

\begin{proof}
	Let $z$ be a point in the coordinate hyperplanes and set $r:= \# J_z$,
	without loss of generality, suppose that $J_z=\{1, \ldots ,r\}$ and $k=r+1$. 
	\begin{enumerate}[(i)]
		\item Suppose that $h(z)\neq 0$.  Take $U_z$ small enough such that $h$ and  $\{  x\mapsto x_j\ ;\ j= r+1,\ldots , n\}$ do not vanish in $U_z$. Now take:
		$\eta_z: x \mapsto x+z $, $\widetilde{U}_z:= \eta_z^{-1} (U_z)$ and $\tilde{h}(x) := \prod_{j=r+1}^n (x_j+z_j)^{\alpha_j} ( h\circ\eta_z)$. Then $ \tilde{h}$ does not vanish on 
		$\widetilde{U}_z$, 
		\[
		g\circ \eta_z(x) =\prod_{j=1}^r {x}^{\alpha_j}_j\tilde{h}(x),
		\]
		and the determinant of the Jacobian of $\eta_z$ is one.
		\item Suppose that $h(z)=0$. Then, by Proposition \ref{No se anula en cero mas no degenerada implica transversal}, the varieties 
		\[V(x_1),\ldots , V(x_r), V(h)\] intersect transversally. By Proposition \ref{Si hay transversalidad hay un buen cambio de coordenadas}, there exists a neighbourhood $U$ of $z$, a neighbourhood ${U}'$ of the origin, and a diffeomorphism  
		\[
		\eta=(\eta_1,\ldots ,\eta_n): ({U}',\underline{0})\longrightarrow (U,z)
		\]
		such that 
		\begin{equation}\label{h se convierte en x a la r+1}
		\eta_i(x)=x_i\quad\text{for}\quad i=1,\ldots ,r\quad\text{and}\quad h\circ \eta(x)=x_{r+1}.
		\end{equation}
		Since $\eta(\underline{0})=z$, then $\eta_i(\underline{0})\neq 0$ for $i=r+1,\ldots n$, so, we may choose $\widetilde{U}_z\subset U'$ to be a small enough neighbourhood of $\underline{0}$ such that $\eta_i(x)\neq 0$ for all $x\in \widetilde{U}_z$. Set $U_z:=\eta(\widetilde{U}_z)$ and $\tilde{h}(x) := \prod_{j=r+1}^n {\left(\eta_j(x)\right)}^{\alpha_j}$. Then $ \tilde{h}$ does not vanish on 
		$\widetilde{U}_z$ and 
		\[
		g\circ \eta (x)=\prod_{j=1}^r {\left(\eta_j(x)\right)}^{\alpha_j} h\circ\eta(x)\tilde{h}(x)\stackrel{\eqref{h se convierte en x a la r+1}}{=}\prod_{j=1}^r x^{\alpha_j}_j x_{r+1} \tilde{h}(x)
		\]
		where $\tilde{h}$ does not vanish on $\widetilde{U}_z$.
	\end{enumerate}
\end{proof}

\begin{cor}\label{Corolario Los entornos que necesitamos para las buenas coordenadas}
	Let $f$ be a polynomial and assume that $f$ is non--degenerate with respect to its Newton polyhedron. Let $\Sigma$ be a simplicial fan subordinated to $f$ and take $\sigma= Cone (\N)\in\Max(\Sigma)$. With the notation of the previous Proposition we have that one of the following holds:
	\begin{enumerate}
		\item $f\circ \psi_{\N\trans}\circ \eta_z =\prod_{j\in J_z} {x_j}^{\alpha_j}\tilde{h}$ where $\tilde{h}(x)\neq 0$ for all $x\in \widetilde{U}_z$.
		\item $f\circ \psi_{\N\trans}\circ \eta_z =\prod_{j\in J_z} {x_j}^{\alpha_j}x_k\tilde{h}$ where $\tilde{h}(x)\neq 0$ for all $x\in \widetilde{U}_z$.
	\end{enumerate}
	In addition, the determinant of the Jacobian of $\eta_z$ is a $\mathcal{C}^\infty(K)-$function not vanishing on $\widetilde{U}_z$ and $\alpha_i=\nu_{u^{(i)}} (f)$, where $u^{(i)}$ stands for the $i$-th column of $\N$.
\end{cor}

	\section{Local Zeta Functions}

\subsection{Some Integrals.}
In this section we present some results about integrals that will be used later on. The first Lemma is about meromorphic continuation of integrals attached
to monomials over the complex numbers. These results are easy variations of the ones presented in \cite[App. B, Sections 2.2 and 2.9]{Gel-Shi}. The real version of this results are presented in e.g. \cite[Lemme 3.1]{Den-Sar} and \cite[Ch. III, Sect. 4.4.]{Gel-Shi}. Next we present a powerful Lemma about not--injective changes of variables in integrals, one may consult \cite[Thm. 1.6.24]{Csikos} for a real version. The same proof for the real version can be adapted to the complex setting. 

\begin{lemma}\label{Monomial Integrals}
	Let $g$ be a polynomial function over $\C$ and let $\phi$ be a smooth function with compact support in some neighbourhood of the origin of $\C^n$. Assume that 
	$g$ has no zeroes in the support of $\phi$.
	Define for $\operatorname{Re}(s)>0$, $m=(m_1,\ldots,m_n)\in\mathbb{N}^n$ and $\nu=(\nu_1,\ldots,\nu_n) \in(\mathbb{N}\setminus\{0\})^n$, the following integral
	\[
	I(s)=\int_{\C^n}\phi(x)\ x^{2sm+\nu-1}|g(x)|^{2s}\ dx.
	\]
	Then the following assertions hold:
	\begin{enumerate}
		\item $I(s)$ is convergent and defines a holomorphic function on
		\[
		\operatorname{Re}(s)>\max\{-1,-\nu_{1}/2m_{1},\dots,-\nu_{n}/2m_{n}\};
		\]
		\item $I(s)$ admits a meromorphic continuation to the whole complex
		plane, with poles of order at most $n$. Furthermore, the poles belong to
		\[
		\bigcup_{1\leq i\leq n}\left(  -\frac{\nu_{i}+\mathbb{N}}{2m_{i}}\right)
		\cup\left(  -\frac{1+\mathbb{N}}{2}  \right)  .
		\]
		\item Let $\kappa$ be a positive integer and let $s_{0}$ be a candidate
		pole of $I(s)$ with $s_{0}\notin-\left(  1+\mathbb{N}\right)/2$ (resp. $s_{0}\in-\left(  1+\mathbb{N}\right)/2$). A necessary condition for
		$s_{0}$ to be a pole of $I(s)$ of order $\kappa$,  is that
		\[
		Card\left\{  i\ ;\  s_{0}\in -\frac{\nu_{i}+\mathbb{N}}{2m_{i}}\right\}  \geq\kappa  \ \text{(resp.}\  \geq\kappa -1\text{)} .
		\]
	\end{enumerate}
\end{lemma}
\begin{lemma}\label{Cambio de variables para recubrimientos}
	Suppose that $U$ and $V$ are open subsets of $\mathbb{C}^n$ and $\Psi: U\to V$ is a map of class
	$\mathcal{C}^{\infty}$. For a point $v \in V$ , denote by $\#\Psi^{-1}(v)$ the number of $\Psi-$preimages of $v$.
	Then for any $\mathcal{C}^{\infty}$-function $g : V \to \C$ we have
	\[\int_Ug(\Psi(u))\ |\det \Psi^\prime (u)|\ du=\int_V g(v)\#\Psi^{-1}(v)\ dv,\]
	provided that both integrals exist. The integrals may not exist, but if any of
	them exists, then the other exists as well.
\end{lemma}

\subsection{Poles of Complex Local Zeta Functions} Recall that for a given $\omega\in\Rnnoneg$, the \textit{$\omega$-order} of $f(x)$ is defined as $\nu_{\omega}(f):=min\{\omega \cdot \mu\ ;\ \mu\in \varepsilon(f)\}.$
Now, for any $u=( u_{1},\ldots,u_{n})  \in\mathbb{N}^{n}
\setminus\left\{  \underline{0}\right\}  $ satisfying $\nu_{u}(f)  \neq0$, we define the following arithmetic progression
\[
\mathcal{P}(u)  =\left\{  -\frac{||u||
	+k}{2\nu_{u}(f)} \ ;\  k\in\mathbb{N} \right\}  .
\]
The remoteness of $NP(f)$ (also called by Varchenko the distance from the origin to $NP(f)$) is defined as
\[
\nu_{0}(f) =\min_{u\in\Ver(\Sigma(f))}\left\{  \frac{||u||}{2\nu_u(f)  }\right\}.
\]
From Varchenko's work, we have that the number $\nu
_{0}(f)$ has a nice geometric
interpretation: if $\left(  t_{0},\ldots,t_{0}\right)  $ is the intersection
point of the diagonal $\{\left(  t,\ldots,t\right)  \in\mathbb{R}^{n}\ ;\ t\in\mathbb{R}\}$ with the boundary of $NP(f)$, then $\nu_{0}(f)=1/t_{0}$.

\begin{thm}\label{Main Thm}
	Let $f$ be a polynomial over the complex numbers, satisfying $f(\underline{0})=0$. Let $NP(f)$ be the Newton polyhedron of $f$ and let $\Sigma$ be a simplicial fan subordinated to $f$.  Assume that $f$ is non-degenerate with respect to $NP(f)$, then there exists a neighborhood $\Omega$ of the origin such that, for every  smooth function $\phi$ with compact support contained in $\Omega$, the following assertions hold.
	\begin{enumerate}
		\item The function $Z_{\phi}(s,f)$ is holomorphic on the complex half-plane $\operatorname{Re}(s)>\max\left\{  -\nu_{0}(f),-1/2\right\}$.
		\item The poles of $Z_{\phi}(s,f)$ belong to the set 
		\[
		\bigcup_{u\in \Ver(\Sigma)}\mathcal{P}(u)  \cup\left(-\frac{1+\mathbb{N}}{2}\right).  
		\]
		\item Let $\kappa$ be an integer satisfying $1\leq\kappa\leq n$, and let $s_{0}$ be a candidate pole of $Z_{\phi}\left(  s,f\right)$ with $s_{0}\notin-(1+\mathbb{N})/2  $ (respectively $s_{0}\in-(1+\mathbb{N})/2  $). A necessary condition for $s_0$ to be a pole of  $Z_{\phi}\left(  s,f\right)$ of order $\kappa$, is that there exists a face $F\subset NP(f)$ of codimension $\kappa$ (respectively of codimension $\kappa-1$) such that $s_{0}\in\mathcal{P}(u)$ for any facet $F_u$ containing $F$.
	\end{enumerate}
\end{thm}

\begin{proof}
	We will use the notation of Section \ref{Seccion NP and non deg}. Let $\Sigma(f)$ be the dual fan of $f$ and take $\Sigma$ a simplicial fan subordinated to $f$, i.e.  such that $\Ver(\Sigma)=\Ver(\Sigma (f))$, see Remark \ref{Refining fans}. Let $T_{\Sigma}$ be the toric variety associated to $\Sigma$ and let $\pi :T_\Sigma\longrightarrow \C^n$ be the toric modification associated to $\Sigma(f)$. The restriction
	\[
	\pi: {T_{\Sigma}}^*:=T_{\Sigma}\setminus \pi^{-1} ( V(x_1\cdots x_n))  \longrightarrow   {(\C^*)}^n
	\]
	is a diffeomorphism. By using $\pi: {T_\Sigma}^*\rightarrow   {(\C^*)}^n $ as a change of variables in the integral
	\[
	Z_{\phi}(s,f)=\int\limits_{({\mathbb{C}^*)}^{n}\smallsetminus \{f^{-1}(0)  \}}\phi\left(  x\right)  \left\vert f(x)\right\vert ^{2s}|dx|,
	\]
	we have
	\[
	Z_{\phi}(s,f) = \int\limits_{ {T_\Sigma}^*\smallsetminus \{\pi^{-1}(f^{-1}(0))  \} }  (\phi\circ \pi)\left( t\right)  \left\vert (f\circ \pi)(t)\right\vert^{2s}\ |J_\pi(t)|\ |dt|,
	\]
	where $|dt|$ is a volume element in $ {T_\Sigma}^*$ and $J_\pi$ denotes the Jacobian matrix of $\pi$.
	
	Since $\pi :T_\Sigma\longrightarrow \C^n$  is proper, the support of $\phi\circ \pi$ as a function over $T_\Sigma$ is compact. We will take $\{ \xi_\sigma \}_{\sigma\in\Max (\Sigma)}$  a $\mathcal{C}^\infty-$partition of the unity of $T_\Sigma$ subordinated to the covering $\{ T_\sigma\}_{\sigma\in\Max (\Sigma)}$ (see subsection \ref{Particiones de unidad}). Then
	\begin{equation}\label{localizacion de Z}
	Z_{\phi}(s,f) = \sum_{\sigma\in\Max (\Sigma)}\int_{{T_\Sigma}^\ast\smallsetminus \{\pi^{-1}(f^{-1}(0))\}}  (\phi\circ\pi)(t)\xi_\sigma(t)\ |(f\circ\pi)(t)|^{2s}\ |J_\pi(t)|\ |dt|.
	\end{equation}
	Note that the function $\vartheta_\sigma := (\phi\circ \pi) \xi_\sigma$, is a $\mathcal{C}^\infty-$function with compact support contained in the chart $T_\sigma$. For every $\sigma\in\Max (\Sigma )$, we take
	\begin{equation*}
	Z_{\phi}(s,f)^\sigma : =  	\int\limits_{ {T_\sigma}^*\smallsetminus \{\pi_\sigma^{-1}(f^{-1}(0))  \} }  \vartheta_\sigma (t)\ |(f\circ \pi_\sigma)(t)|^{2s}\ |J_{\pi_\sigma}(t)|\ |dt|.
	\end{equation*}
	In view of this notation, \eqref{localizacion de Z} becomes:
	\begin{equation}\label{decomposition of Z}
	Z_{\phi}(s,f) = \sum_{\sigma\in\Max (\Sigma)}Z_{\phi}(s,f)^\sigma.
	\end{equation}
	Now our task is to compute $Z_{\phi}(s,f)^\sigma$. To do so, take a cone $\sigma =Cone (\N )\in\Max(\Sigma)$ and set $\M$ to be the matrix with $\sigma^\vee =\Cone (\M)$, and such that $\Lambda := \N\trans \M$ is diagonal. Recall from Definition \ref{La parametrizacion de la carta asociada a sigma}, that $\Phi_\sigma: \C^n\longrightarrow T_\sigma$ is given by $\Phi_\sigma (x) = (\psi_\Lambda (x), \psi_{\N\trans}(x), \psi_{{\N\trans} W} (x))$. The restriction of $\Phi_\sigma$ to $(\C^\ast)^n$ is a finite (locally diffeomorphism) covering, lets say of degree $d$. By Lemma \ref{Cambio de variables para recubrimientos},
	\begin{align*}
	&Z_{\phi}(s,f)^\sigma  =\\
	& \frac{1}{d}\int\limits_{ {(\C^*)}^n\smallsetminus \{{(f\circ\pi_\sigma\circ\Phi_\sigma)}^{-1}(0)  \}} \vartheta_\sigma (\Phi_\sigma (x))\ |(f\circ \pi_\sigma)(\Phi_\sigma (x))|^{2s}\ |J_{\pi_\sigma} (\Phi_\sigma (x))|\ |J_{\Phi_\sigma} (x)|\ |d  x|\\
	& = \frac{1}{d}
	\int\limits_{ {(\C^*)}^n\smallsetminus \{{(f\circ\pi_\sigma\circ\Phi_\sigma)}^{-1}(0)  \} }  (\vartheta_\sigma \circ \Phi_\sigma) (x) \left\vert (f\circ \pi_\sigma\circ \Phi_\sigma) (x)\right\vert^{2s}\ |J_{(\pi_\sigma\circ\Phi_\sigma)} (x)|\ |d  x| \\
	&\stackrel{\eqref{Diagrama de covering}}{=} \frac{1}{d} \int\limits_{ {(\C^*)}^n\smallsetminus \{{(f\circ\psi_{\N\trans})}^{-1}(0)  \} }  (\vartheta_\sigma \circ \Phi_\sigma) (x) \left\vert (f\circ \psi_{\N\trans}) (x)\right\vert^{2s}\ |J_{\psi_{\N\trans}} (x)|\ |d  x|.
	\end{align*}
	By Lemma \ref{Prop:JacdepsiA},
	\begin{align*}
	&Z_{\phi}(s,f)^\sigma \\
	&=\frac{1}{d} \int\limits_{ {(\C^*)}^n\smallsetminus \{{(f\circ\psi_{\N\trans})}^{-1}(0)  \} }  (\vartheta_\sigma \circ \Phi_\sigma) (x) \left\vert (f\circ \psi_{\N\trans}) (x)\right\vert^{2s}\ |\det \N\trans| \prod_{i=1}^{n} |x_i|^{||u^{(i)}||-1} \ |dx|\\
	&= \frac{1}{d} \int\limits_{ {(\C^*)}^n\smallsetminus \{{(f\circ\psi_{\N\trans})}^{-1}(0)  \} }  \tilde{\vartheta}_\sigma(x) \left\vert (f\circ \psi_{\N\trans}) (x)\right\vert^{2s}\ \prod_{i=1}^{n} |x_i|^{||u^{(i)}||-1}\ |dx|,
	\end{align*}
	where $\tilde{\vartheta}_\sigma :=  |det \N\trans| (\vartheta_\sigma\circ\Phi_\sigma)$. Since $\vartheta_\sigma$ has compact support and $\Phi_\sigma$ is a smooth and finite covering, $\tilde{\vartheta}_\sigma: \C^n\longrightarrow \C$ is a smooth function with compact support. 
	
	In the next step we use the knowledge that we have about the term $f\circ\psi_{\N\trans}$. Consider a point $z\in V(x_1\cdots x_n)\cap Supp (\tilde{\vartheta}_\sigma)$ and let $U_z$ be a neighbourhood of $z$ as in Corollary \ref{Corolario Los entornos que necesitamos para las buenas coordenadas}. Note that 
	\[V(x_1\cdots x_n)\cap Supp (\tilde{\vartheta}_\sigma)\subset  \bigcup_{z\in V (x_1\cdots x_n)\cap Supp (\tilde{\vartheta}_\sigma)}U_z, \] which implies, since $Supp (\tilde{\vartheta}_\sigma)$ is compact, that there exists a finite subset 
	$W\subseteq V (x_1\cdots x_n)\cap Supp (\tilde{\vartheta}_\sigma)$, with
	\[V(x_1\cdots x_n)\cap Supp (\tilde{\vartheta}_\sigma)\subset  \bigcup_{z\in W}U_z.\]
	By shrinking $Supp (\phi)$ if necessary, we may assume that $Supp (\tilde{\vartheta}_\sigma)\subset \bigcup_{z\in W} U_z$ and then we have to deal with integrals of type $Z_\phi(s,f)^\sigma$ over $U_z$. 
	
	Since $U_z$ has been chosen as in Corollary \ref{Corolario Los entornos que necesitamos para las buenas coordenadas}, there exist: a set of indices $J_z\subseteq\{1,\ldots,n\}$, a marked index $k\in\{1,\ldots,n\}\setminus J_z$, a neighbourhood $\widetilde{U}_z$ of $\underline{0}$ and a diffeomorphism $\eta_z: (\widetilde{U}_z,\underline{0})\longrightarrow  (U_z,z)$ such that 
	$$f\circ \psi_{\N\trans}\circ \eta_z =\prod_{j\in J_z} {x_j}^{\alpha_j}\tilde{h}\quad\text{ or }\quad f\circ \psi_{\N\trans}\circ \eta_z =\prod_{j\in J_z} {x_j}^{\alpha_j}x_k\tilde{h},$$ where $\tilde{h}(x)\neq 0$ for all $x\in \widetilde{U}_z$.
	
	By using $\eta_z$ as a change of variables in our integral $Z_\phi(s,f)^\sigma$ over $U_z$, we get that $Z_\phi(s,f)^\sigma$ becomes a finite sum of integrals of types $I_1(s)$ and $I_2(s)$, where
	\begin{align*}
	&I_1(s)=\\
	&\bigintss\limits_{\widetilde{U}_z}
	\Biggl| \Biggl(\prod_{j\in J_z} {x_j}^{\nu_{u^{(j)}}(f)}\!\!\Biggr)\tilde{h}(x)\Biggr|^{2s} \!\!\!\prod_{j\in J_z}\!\! |x_j|^{||u^{(j)}||-1}\!\! \prod_{j\notin J_z} \!\!|{(\eta_z)}_j|^{||u^{(j)}||-1}
	{\theta}_{\sigma ,z}(x)\ |J_{\eta_z}(x)||d  x|,
	\end{align*}
	and
	\begin{align*}
	&I_2(s)=\\
	&\bigintss\limits_{\widetilde{U}_z}
	\Biggl| \Biggl(\prod_{j\in J_z} {x_j}^{\nu_{u^{(j)}}(f)}\!\!\Biggr)x_l\tilde{h}(x)\Biggr|^{2s} \!\!\!\!\prod_{j\in J_z} \!\! |x_j|^{||u^{(j)}||-1}\!\! \prod_{j\notin J_z} \!\!|{(\eta_z)}_j|^{||u^{(j)}||-1}
	{\theta}_{\sigma ,z}(x)\ |J_{\eta_z}(x)||d  x|.
	\end{align*}
	As the reader may have guessed, we have set $\theta_{\sigma ,z}=\tilde{\vartheta}_\sigma\circ\eta_z$. Finally, note that
	\begin{equation}\label{Final 1}
	I_1(s)=\bigintsss\limits_{\widetilde{U}_z}
	\prod_{j\in J_z} |x_j|^{2s\nu_{u^{(j)}}(f)+||u^{(j)}||-1}|\tilde{h}(x)|^{2s} \tilde{\theta}_{\sigma ,z}(x)\ |d  x|,
	\end{equation}
	and
	\begin{equation}\label{Final 2}
	I_2(s)=\bigintsss\limits_{\widetilde{U}_z}
	\prod_{j\in J_z} |x_j|^{2s\nu_{u^{(j)}}(f)+||u^{(j)}||-1}|x_k\tilde{h}(x)|^{2s} \tilde{\theta}_{\sigma ,z}(x)\ |d  x|.
	\end{equation}
	
	All the assertions of our theorem follows now from \eqref{decomposition of Z} and Lemma \ref{Monomial Integrals} applied to \eqref{Final 1} and \eqref{Final 2}.
\end{proof}
\begin{rmk}\label{Weakening the non degeneracy condition}
	By shrinking $\Omega$ in the hypothesis of Theorem \ref{Main Thm}, if necessary, one may just ask for $f$ to be non-degenerate with respect to the \textit{compact} faces of $NP(f)$.
\end{rmk}

\bibliographystyle{plain}

\end{document}